\documentclass[11pt]{amsart}
\usepackage{graphicx, color}
\usepackage{amscd}
\usepackage{amsmath,empheq}
\usepackage{amsfonts}
\usepackage{amssymb}
\usepackage{mathrsfs}

\usepackage[all]{xy}

\usepackage{todonotes}
\usepackage{enumitem}

\newtheorem{theorem}{Theorem}
\newtheorem{remark}[theorem]{Remark}
\newtheorem{lemma}[theorem]{Lemma}
\newtheorem{proposition}[theorem]{Proposition}

\newtheorem{definition}[theorem]{Definition}

\DeclareMathOperator*{\divergenz}{div}              %
\DeclareMathOperator*{\essinf}{ess ~inf}         %
\DeclareMathOperator*{\Ss}{S}         %

\newcommand{\N}{\mathbb{N}}

\newcommand{\R}{\mathbb{R}}

\newcommand{\Lp}[1]{L^{#1}(\Omega)}

\newcommand{\into}{\int_{\Omega}}

\newcommand{\close}{\overline{\Omega}}

\numberwithin{theorem}{section}
\numberwithin{equation}{section}

\title[Double phase problems for the Baouendi-Grushin operator]{Nonvariational and singular double phase \\
problems for the Baouendi-Grushin operator }

\author[A. Bahrouni]{Anouar Bahrouni}
\address[A. Bahrouni]{Mathematics Department, University of Monastir,
Faculty of Sciences, 5019 Monastir, Tunisia}
\email{bahrounianouar@yahoo.fr}

\author[V.D. R\u{a}dulescu]{Vicen\c{t}iu D. R\u{a}dulescu}\thanks{Corresponding author: Vicen\c tiu D. R\u adulescu ({\tt radulescu@inf.ucv.ro}).}
\address[V.D. R\u{a}dulescu]{Faculty of Applied Mathematics, AGH University of Science and Technology, al. Mickiewicza 30, 30-059 Krak\'{o}w, Poland  \& Department of Mathematics, University of Craiova, 200585 Craiova, Romania \& `Simion Stoilow' Institute of Mathematics of the Romanian Academy, P.O. Box 1-764, 014700 Bucharest, Romania}
\email{\tt radulescu@inf.ucv.ro}

\author[D.D. Repov\v{s}]{Du\v{s}an D. Repov\v{s}}
\address[D.D. Repov\v{s}]{Faculty of Education and Faculty of Mathematics and Physics, University of Ljubljana \& Institute of Mathematics, Physics and Mechanics, 1000 Ljubljana, Slovenia}
\email{\tt dusan.repovs@guest.arnes.si}


\begin{document}
  \begin{abstract}
In this paper we introduce a new double phase Baouendi-Grushin type
operator with variable coefficients. We give basic properties
of the corresponding functions space and prove a compactness result.
In the second part, using topological argument, we prove the
existence of weak solutions of some nonvariational problems in which
this new operator is present. The present paper extends and complements some of our
previous contributions related to double phase anisotropic
variational integrals.
  \end{abstract}

\maketitle

\section{Introduction}\label{section_introduction}
The present paper is motivated by recent  fundamental enrichment to
the mathematical analysis of nonlinear models with unbalanced
growth. We mainly refer to the pioneering contributions of
Marcellini \cite{marce1,marce2} who studied lower semicontinuity and
regularity properties of minimizers of certain quasiconvex
integrals. Related problems are inspired by models arising in
nonlinear elasticity and they describe the deformation of an elastic
body, see Ball \cite{ball1,ball2}.\\

 More precisely, we are concerned
with the following nonlinear
 equations of double phase Baouendi-Grushin type
 \begin{equation}\label{imain}
-\Delta_{G,a}u+|u|^{G(z)-2}u=K(z)f(u), \ \ z\in \mathbb{R}^{N},
 \end{equation}
 where $N\geq 3$, $ K\in C(\mathbb{R}^{N})$, $f\in C(\mathbb{R})$,
 while $-\Delta_{G,a}$ stands for a new double phase
 Baouendi-Grushin type operator with variable exponents (see \eqref{baouendi}).\\

 The main aim of our work is to introduce a new  double phase
 Baouendi-Grushin type operator with variable exponents and its suitable functions space. Our
abstract results related to the new function space are motivated by
the existence of solutions for nonvariational problems of type
\eqref{imain}. The present paper complements our previous
contributions related to double phase
anisotropic variational integrals, see \cite{Bahrouni-Radulescu-Repovs-2018,Bahrouni-Radulescu-Repovs-2019,Bahrouni-Radulescu-Winkert-2019, Bahrouni-Radulescu-2020}.\\

First, we recall the notion of Baouendi-Grushin operator with variable growth.
Let $\Omega \subset \R^N$, $N>1$, be a  domain with smooth boundary
$\partial\Omega$ and let $n,m$ be nonnegative integers such that
$N=n+m$. This means that $\R^N=\R^n\times\R^m$ and so $z \in \Omega$
can be written as $z=(x,y)$ with $x \in \R^n$ and $y\in\R^m$. In
this paper $G:\close \to (1,\infty) $ is supposed to be a continuous
function and $\Delta_{G(x,y)}$ stands for the Baouendi-Grushin
operator with variable coefficient, which is defined by
\begin{align*}
    \Delta_{G(x,y)}u
    &=\divergenz\left(\nabla_{G(x,y)}u\right)\\
    &=\sum_{i=1}^n \left(|\nabla_xu|^{G(x,y)-2}u_{x_i}\right)_{x_i}+|x|^\gamma \sum_{i=1}^m \left(|\nabla_yu|^{G(x,y)-2}u_{y_i}\right)_{y_i},
\end{align*}
where
\begin{align*}
    \nabla_{G(x,y)}u=\mathcal{A}(x)
    \begin{bmatrix}
    |\nabla_xu|^{G(x,y)-2} & \nabla_x u\\[2ex]
    |x|^\gamma |\nabla_yu|^{G(x,y)-2} & \nabla_y u
    \end{bmatrix}
\end{align*}
and
\begin{align*}
    \mathcal{A}(x)=
    \begin{bmatrix}
    I_n & 0_{n,m}\\[1ex]
    0_{m,n} & |x|^\gamma I_m
    \end{bmatrix}\in \mathcal{M}_{N\times N}(\R),
\end{align*}
with $I_n$ being the identity matrix of size $n\times n$, $O_{n,m}$
is the zero matrix of size $n\times m$ and $\mathcal{M}_{N\times N}$
stands for the class of $(N\times N)$--matrices with real-valued
entries. From the representation above it is clear that
$\Delta_{G(x,y)}$ is degenerate along the $m$-dimensional subspace
$M:=\{0\}\times\R^m$ of $\R^N$.

The differential operator $\Delta_{G(x,y)}$ generalizes the
degenerate operator
$$\frac{\partial^2}{\partial x^2}+x^{2r}\frac{\partial^2}{\partial y^2}\quad (r\in{\mathbb N})$$ introduced
independently by
 Baouendi \cite{baouendi} and Grushin \cite{grushin}.
The Baouendi--Grushin operator can be viewed as the Tricomi operator
for transonic flow restricted to subsonic regions. On the other
hand, a second-order differential operator $T$ in divergence form on
the plane, can be written as an operator whose principal part is a
Baouendi-Grushin-type operator, provided that the principal part of
$T$ is nonnegative and its quadratic form does not vanish at any
point, see Franchi \& Tesi \cite{franchi}. For recent contributions
to the study of double-phase problems we cite Beck \& Mingione
\cite{beck},  { Eleuteri, Marcellini \& Mascolo \cite{Eleuteri}}, Papageorgiou, R\u adulescu \& Repov\v{s} \cite{prrzamp,
prrpams},  {Pucci {\it et al.} \cite{Pucci0, Pucci1}}, and Zhang \& R\u adulescu \cite{zhang}.  {We refer to Marcellini \cite{Marce} and Mingione \& R\u adulescu \cite{Mingi} for surveys of recent results on elliptic variational problems with nonstandard growth conditions and related to different kinds of nonuniformly elliptic operators.} \\

Now, we are able to introduce the new Baouendi-Grushin type operator
with variable coefficients, which is defined by
\begin{align}\label{baouendi}
    \Delta_{G,a}u
    &=\divergenz\left(\nabla_{G(x,y)}u\right)\\
    &=\sum_{i=1}^n \left(|\nabla_xu|^{G(x,y)-2}u_{x_i}\right)_{x_i}+a(x) \sum_{i=1}^m
    \left(|\nabla_yu|^{G(x,y)-2}u_{y_i}\right)_{y_i}\nonumber.
\end{align}

The main goal of our recent paper \cite{Bahrouni-Radulescu-2020} was
to study a singular systems in the whole space $\mathbb{R}^{N}$ in which
the Baouendi-Grushin operator ($-\Delta_{G(x,y)})$ is present. So,
the main difficulty is the lack of compactness corresponding to the
whole Euclidean space. To overcome this difficulty, we proved a related compactness
property. However, the interval of compactness is too short. So, we are not
 able to study a large number of equations driven by
$-\Delta_{G(x,y)}$ in the whole space $\mathbb{R}^{N}$.
 For this reason and in order to get a better compactness result, we
 introduced
the new operator $-\Delta_{G,a}$.   {Our abstract
results are motivated by the existence of solutions of the following
class of nonlinear equation}
\begin{equation}\label{int1}
-\Delta_{G,a}u=-\mbox{div}(\alpha_1
u\nabla_{x}r)-\mbox{div}(\alpha_2a^{\frac{1}{G(x,y)}}(x)u\nabla_{y}r)+f(z,u),
\ \ z=(x,y)\in \mathbb{R}^{N},
\end{equation} where $\Omega\subset \mathbb{R}^{N}$ is supposed to be a bounded
domain. {   Another motivation comes from singular problems
in the form}
\begin{equation}\label{int2}
-\Delta_{G,a}u+|u|^{G(x,y)-2}u=\frac{b(x,y)}{u^{\sigma(x,y)}}, \ \
(x,y)\in \mathbb{R}^{N},
\end{equation}
where $\sigma(\cdot) \in (0,1)$ and $b$ is positive function. \\

The paper is organized as follows. In Section $2$ we present the
basic properties of variable Lebesgue space and introduce the main tools
which will be used later. New properties concerning the new operator
($-\Delta_{G,a}$) will be discussed in Section $3$.
 In Section $4$, combining these abstract results with the
 topological argument, we study a nonvariational problem in which
 $-\Delta_{G,a}$ is present. In last section, we deal with purely
 singular double phase equation.
 We refer to the monograph by  Papageorgiou,  R\u adulescu \& Repov\v{s} \cite{prrbook} as a general reference for the abstract methods used in this paper.
\section{Terminology and the abstract setting}\label{section_preliminaries}

In this section we recall some necessary definitions and properties
of variable exponent spaces. We refer to the papers of Bahrouni \&
Repov\v{s} \cite{Bahrouni-Repovs-2018}, H\'{a}jek, Montesinos
Santaluc\'{\i}a, Vanderwerff \& Zizler
\cite{Hajek-Montesinos-Santalucia-Vanderwerff-Zizler-2008}, Musielak
\cite{Musielak-1983}, R\u{a}dulescu \cite{Radulescu-2015, rad2019},
R\u{a}dulescu \& Repov\v{s} \cite{Radulescu-Repovs-2015} and the
references therein. Consider the set
\begin{align*}
    C_+(\overline\Omega)=\left\{p\in C(\close)\ \bigg| \ p(x)>1 \ \text{for all } x\in\close\right\}
\end{align*}
and define for any $p\in C_+(\close)$
\begin{align*}
    p^+:=\sup_{x\in\close} p(x)\qquad
    \text{and}\qquad
    p^-:= \inf_{x\in\close}p(x).
\end{align*}
Then $1<p^-\leq p^+<\infty$ for each $p\in C_+(\close)$. The
variable exponent Lebesgue space $\Lp{p(\cdot)}$ is defined by
\begin{align*}
    \Lp{p(\cdot)}=\left\{u\colon \Omega \to \R \ \bigg | \ u\text{ is measurable and} \int_\Omega|u(x)|^{p(x)}\,dx<\infty\right\}
\end{align*}
equipped with the Luxemburg norm
\begin{align*}
    \|u\|_{p(\cdot),\Omega}=\inf\left\{\mu>0 \ \bigg | \ \into\left|\frac{u(x)}{\mu}\right|^{p(x)}\,dx\leq 1\right\}.
\end{align*}
If $\Omega=\mathbb{R}^{N}$, we denote $\|u\|_{p(\cdot),\Omega}=
\|u\|_{p(\cdot)}$.

 It is well known that $\Lp{p(\cdot)}$ is a reflexive
Banach space.

Let $L^{q(x)}(\Omega)$ denote the conjugate space of
 $L^{p(x)}(\Omega)$, where $1/p(x)+1/q(x)=1$. If $u\in
 L^{p(x)}(\Omega)$ and $v\in L^{q(x)}(\Omega)$ then  the following
 H\"older-type inequality holds:
 \begin{equation*}
 \left|\int_\Omega uv\;dx\right|\leq\left(\frac{1}{p^-}+
 \frac{1}{q^-}\right)\|u\|_{p(.)}\|v\|_{q(.)}\,.
 \end{equation*}
 Also, if $p_j\in C_+(\overline\Omega)$ ($j=1,2,\ldots, k$) and
 $$\frac{1}{p_1(x)}+\frac{1}{p_2(x)}+\cdots +\frac{1}{p_k(x)}=1,$$
 then for all $u_j\in L^{p_j(x)}(\Omega)$ ($j=1,\ldots ,k$) we have
 \begin{equation}\label{Hol1}
 \left|\int_\Omega u_1u_2\cdots u_k\;dx\right|\leq\left(\frac{1}{p_1^-}+
 \frac{1}{p_2^-}+\cdots +\frac{1}{p_k^-}\right)|u_1|_{p_1(x)}|u_2|_{p_2(x)}\cdots |u_k|_{p_k(x)}\,.
 \end{equation}

Moreover, if $p_1\leq p_2$ in $\Omega$ and $\Omega$ has finite Lebesgue measure, then there exists the
continuous embedding
\begin{equation}\label{inj1}
L^{p_2(\cdot)}(\Omega)\hookrightarrow L^{p_1(\cdot)}(\Omega).
\end{equation}

The following two propositions will be useful in the sequel, see R\u adulescu \& Repov\v{s} \cite[p. 11]{Radulescu-Repovs-2015}

\begin{proposition}\label{pr1}
    Let
    \begin{align*}
        \rho_1(u)=\displaystyle \int_{\Omega}|u|^{p(x)}\,dx \quad \text{for all }u\in L^{p(\cdot)}(\Omega).
    \end{align*}
    Then the following hold:
    \begin{enumerate}
    \item[(i)]
        $\|u\|_{p(\cdot),\Omega}<1 \ (\text{resp}., =1;>1)$ if and only if $\rho_1(u)<1 \ (\text{resp}., =1;>1)$;
    \item[(ii)]
        $\|u\|_{p(\cdot),\Omega}>1$ implies $\|u\|_{p(\cdot),\Omega}^{p^{-}}\leq \rho_1(u) \leq \|u\|_{p(\cdot),\Omega}^{p^{+}}$;
    \item[(iii)]
        $\|u\|_{p(\cdot),\Omega}<1$ implies $\|u\|_{p(\cdot),\Omega}^{p^{+}}\leq \rho_1(u) \leq \|u\|_{p(\cdot),\Omega}^{p^{-}}$.
    \end{enumerate}
\end{proposition}

\begin{proposition}\label{pr2}
    Let
    \begin{align*}
        \rho_1(u)=\displaystyle \int_{\Omega}|u|^{p(x)}\,dx \quad \text{for all }u\in L^{p(\cdot)}(\Omega).
    \end{align*}
    If $u,u_{n}\in L^{p(\cdot)}(\Omega)$ and  $n\in \N$, then the
    following statements are equivalent:
    \begin{enumerate}
    \item[(i)]
        $\displaystyle \lim_{n\to +\infty} \|u_{n}-u\|_{p(\cdot),\Omega}=0$;\\
    \item[(ii)]
        $\displaystyle \lim_{n\to +\infty} \rho_1( u_{n}-u)=0$;\\
    \item[(iii)]
        $u_{n}(x)\to u(x)$ in $\Omega$ and $\displaystyle \lim_{n\to +\infty}\rho_1(u_{n})=\rho_1(u)$.
    \end{enumerate}
\end{proposition}

In what follows, we recall Lemma A.1 of Giacomoni, Tiwari \& Warnault \cite{21} for variable
exponent Lebesgue spaces which is necessary to verify the coercivity
in Section $4$. A related property can be found in Edmunds \& R\'akosnik \cite[Lemma 2.1]{Edm}.

\begin{lemma}\label{L2.2}
Assume that $h_{1}\in L^{\infty}(\Omega)$ such that $h_{1}\geq0$ and
$h_{1}\not\equiv0$ a.e. in $\Omega$. Let
$h_{2}:\Omega\rightarrow\mathbb{R}$ be a measurable function such
that $h_{1}h_{2}\geq1$ a.e. in $\Omega$. Then for any $u\in
L^{h_{1}(\cdot)h_{2}(\cdot)}(\Omega)$,
\begin{equation*}
\||u|^{h_{1}(\cdot)}\|_{h_{2}(\cdot)}\leq\|u\|^{h_1^-}_{h_1(\cdot)h_2(\cdot)}+\|u\|^{h_1^+}_{h_1(\cdot)h_2(\cdot)}.
\end{equation*}
\end{lemma}

Next, we define the variable exponent  Sobolev space
$$
W^{1,p(\cdot)}(\Omega)=\{u\in L^{p(\cdot)}(\Omega):\;|\nabla u|\in
L^{p(\cdot)} (\Omega) \}.
$$
On $W^{1,p(\cdot)}(\Omega)$ we may consider one of the following
equivalent norms
$$
\|u\|_{W}=\|u\|_{p(\cdot)}+\|\nabla u\|_{p(\cdot)}
$$
or
$$
\|u\|_{W}=\inf\left\{\mu>0;\;\int_\Omega\left(\left| \frac{\nabla
u(x)}{\mu}\right|^{p(x)}+\left|
\frac{u(x)}{\mu}\right|^{p(x)}\right)\;dx\leq 1\right\}\,.
$$
We also define $W_0^{1,p(\cdot)}(\Omega)$ as the closure of
$C_0^\infty(\Omega)$ in $W^{1,p(.)}(\Omega)$.

 Next, we recall an
embedding result regarding variable exponent Sobolev spaces, see Fan, Shen \& Zhao
\cite{fan}.
\begin{theorem}\label{binj}
If $\Omega \subset \mathbb{R}^{N}$ is bounded domain and $p(x)\in
C(\overline{\Omega})$, then for any measurable function $q(x)$
defined in $\Omega$ with
$$p(x)\leq q(x), \ \ \mbox{a.e} \ \ x\in \overline{\Omega} \ \ \mbox{and} \ \ \displaystyle \essinf_{x\in \overline{\Omega}}(p^{\ast}(x)-q(x))>0, \ \  (q^{\ast}(\cdot)=\frac{q(\cdot)}{q(\cdot)-1})$$
there is a  compact embedding
$W^{1,p(\cdot)}_{0}(\Omega)\hookrightarrow L^{q(\cdot)}(\Omega).$
\end{theorem}

\section{Double phase Baouendi-Grushin operators}
In this section we prove new results concerning the new
Baouendi-Grushin operator defined in \eqref{baouendi}.

\smallskip
 First, we
give the hypotheses on continuous functions
$a,K,G:\mathbb{R}^{N}\rightarrow \mathbb{R}$.\\
$(A)$ $a(.)$ is a continuous function such that
$$a(x)>0 \ \mbox{for all}\ x\in \mathbb{R}^{N}.$$
 $(G)$ $G$ is a function of class $C^{1}$ and that $$G(x,y)\in
(2,N) \  \mbox{for every}\ (x,y)\in\mathbb{R}^{N}.$$ {   We need
$G>2$ in the proof of Lemma $4.5$, that is, in the first application.  So, it is possible to include the case $G=2$ if we do another kind of applications. }\\
 $(K)$ $K\in L^{\infty}(\mathbb{R}^{N})$,
$K(x)>0$ for all $x\in \mathbb{R}^{N}$ and if $(A_n)\subset
\mathbb{R}^{N}$ is a sequence of Borel sets such that the Lebesgue
measure $|A_n|\leq R,$ for all $n\in \mathbb{N}$ and some $r>0$,
then
$$\displaystyle \lim _{n\rightarrow +\infty}\int_{A_n\cap B^{c}_{r}(0)}K(x)dx=0.$$

In order to treat problem \eqref{imain}, let us consider the space:
\begin{align*}
D^{1,G}_{a}(\mathbb{R}^{N})&=\{u:\mathbb{R}^{N} \rightarrow
\mathbb{R}, \ \ u\in L^{G^{\ast}}(\mathbb{R}^{N})  \ \ and\\
&\int_{\mathbb{R}^{N}}(|\nabla_{x}u|^{G(x,y)}+a(x)|\nabla_{y}u|^{G(x,y)})dxdy<+\infty\}
\end{align*}
endowed with the norm
 $$
  \|u\|_{D}=\left\|\nabla
_{x}u\right\|_{G(\cdot,\cdot)}+\left\| a(x)^{\frac{1}
{G(\cdot,\cdot)}} \nabla_{y}u\right\|_{G(\cdot,\cdot)},\  \mbox{for all}\
u\in X.
$$
This permits us to construct a suitable space
$$X=D^{1,G(\cdot)}_{a}(\mathbb{R}^{N}) \bigcap L^{G(\cdot)}(\mathbb{R}^{N}),$$
endowed with the norm
$$\|u\|_{X}=\|u\|_{D}+\|u\|_{G(\cdot)}\ \mbox{for all}\ u\in X.$$
\begin{remark}
Note that the norm $\|\cdot\|_{X}$ on $X$ is equivalent to
\begin{align}\label{equivalent_norm}
  \begin{split}
    & \|u\|\\
    &=\inf\left\{\mu\geq 0 \ \bigg| \ \rho\left(\frac{u}{\mu}\right)\leq 1\right\}\\
    &=\inf\left\{\mu\geq 0 \ \bigg | \ \displaystyle \int_{\mathbb{R}^{N}}\left[\left|\nabla_{x}\left(\frac{u}{\mu}\right)\right|^{G(x,y)}+
     a(x)\left|\nabla_{y}\left(\frac{u}{\mu}\right)\right|^{G(x,y)}+\left(\frac{|u|}{\mu}\right)^{G(x,y)}\right]\,dx\,dy\leq
     1\right\},
  \end{split}
\end{align}
where \begin{equation} \rho(u)=\displaystyle
\int_{\mathbb{R}^{N}}\left[\left|\nabla_{x}u\right|^{G(x,y)}+
     a(x)\left|\nabla_{y}u\right|^{G(x,y)}+|u|^{G(x,y)}\right]\,dx\,dy.
\end{equation}
\end{remark}
From now on, we shall denote the duality pairing between $X$ and its dual
space $X^*$ by $\langle \cdot,\cdot\rangle_{X}$.\\
 The following
lemma will be helpful in the sequel.

\begin{lemma}\label{modular}
    Suppose that conditions $(A)$ and $(G)$ are satisfied. Let $u\in X$, then the following holds:
    \begin{enumerate}
    \item[(i)]
        For $u\neq 0$ we have: $\|u\|=a$ if and only if $\rho(\frac{u}{a})=1$;
    \item[(ii)]
        $\|u\|<1$ implies $\frac{\|u\|^{G^{+}}}{2^{\frac{1}{G^{+}-1}}}\leq \rho(u)\leq 2 \|u\|^{G^-}$;
    \item[(iii)]
        $\|u\|>1$ implies $\|u\|^{G^-}\leq \rho(u)\leq \|u\|^{G^+} $.
    \end{enumerate}
\end{lemma}
\begin{proof}
The proof is similar to that in
\cite{Bahrouni-Radulescu-Winkert-2019}.
\end{proof}
\begin{lemma}\label{S+}
    Assume that the hypotheses of Lemma \ref{modular} are fulfilled. Then the following properties hold.
    \begin{enumerate}
    \item[(i)]
        The functional $\rho$ is of class $C^{1}$ and for all $u,v\in X$ we have
        \begin{align*}
        \langle \rho'(u),v\rangle_{X} &=\displaystyle \int_{\mathbb{R}^{N}}\left[\left|\nabla_{x}u\right|^{G(x,y)-2}\nabla_{x}u\nabla_{x}v
        +a(x)\left|\nabla_{y}u\right|^{G(x,y)-2} \nabla_{y}u
        \nabla_{y}v\right]\,dx\,dy\\
        &+\displaystyle \int_{\mathbb{R}^{N}} |u|^{G(z)-2}uv\, dz.
        \end{align*}
    \item[(ii)]
        The function $\rho': X\to X^*$ is coercive, that is, $\frac{\langle\rho'(u),u\rangle_{X}}{\|u\|_X}\to +\infty$ as $\|u\|_X\to +\infty$.
        \item[(iii)] $\rho'$
is a mapping of type $(\Ss_+$), that is, if $u_{n}\rightharpoonup u$
in $X$ and $\displaystyle \limsup_{n\to +\infty}\,\langle
\rho'(u_{n}),u_{n}-u\rangle_{X} \leq0$, then $u_{n}\to u$ in $X$.
    \end{enumerate}
\end{lemma}
\begin{proof}
The proof is similar to that in Bahrouni, R\u adulescu \& Winkert
\cite{Bahrouni-Radulescu-Winkert-2019}.
\end{proof}
Now, we establish the following compactness result.
\begin{lemma}\label{compactbound}
Assume that $(A)$ and $(G)$ hold. Then
$D^{1,G}_{a}(\mathbb{R}^{N})$ is compactly embedded in
$L^{s(\cdot)}_{loc}(\mathbb{R}^{N})$, for every $s(\cdot)\in
(1,G^{\ast}(\cdot))$.
\end{lemma}
\begin{proof}
Let $(u_{n})$ be an arbitrary bounded sequence in
$D^{1,G}_{a}(\mathbb{R}^{N})$. Fix $R>0$, $s(\cdot)\in (1,
G^{\ast}(\cdot))$, and set $B(0,R)=\{x\in \mathbb{R}^{N}, |x|\leq
R\}$.\\  We note that $u_n\rightharpoonup u$
 weakly in $L^{G^{\ast}(\cdot)}(\mathbb{R}^{N})$. Thus, for every $\varphi  \in
 C_{0}^{\infty}(\mathbb{R}^{N})$, one has
 \begin{equation}\label{com1}
\displaystyle \lim_{n\rightarrow +\infty}\int_{\mathbb{R}^{N}} u_n
\varphi dx=\int_{\mathbb{R}^{N}} u \varphi dx.
 \end{equation}
\textbf{Claim}. We prove that $u_n\rightharpoonup u$ in
$W^{1,G(\cdot)}_{0}(B(0,R))$. Indeed, denote by
$u\restriction{B(0,R)}$ the restriction of $u$ to $B(0,R)$ and
suppose that $(u_n)$ does not
converge to $u\restriction{B_R}$ weakly in $W^{1,G}_{0}(B(0,R))$. \\
By condition $(A)$, there exists $x_{0}\in B(0,R)$ such that
$$a(x)\geq a(x_0)>0, \ \ \mbox{for all} \ \ x\in B(0,R),$$
and so $(u_n)$  is bounded in $ W^{1,G}_{0}(B(0,R))$. Therefore,
there exist a subsequence $(u_{n_{k}})$ and $\overline{u}\in
W^{1,G}(B(0,R))$, with $\overline{u}\neq u\restriction{B_R}$, such
that $u_{n_{k}} \rightharpoonup \overline{u}$ weakly in
$W^{1,G}_{0}(B(0,R))$. Invoking Theorem \ref{binj}, $u_{n_{k}}
\rightarrow \overline{u}$ strongly in $L^{s(\cdot)}(B(0,R))$. Then,
taking into account \eqref{com1}, we obtain
\begin{equation*}
\int_{B(0,R)} u \varphi dx=\displaystyle \lim_{k\rightarrow
+\infty}\int_{B(0,R)} u_{n_{k}} \varphi dx=\int_{B(0,R)}
\overline{u} \varphi dx,
 \end{equation*}
for every $\varphi \in C_{0}^{\infty}(B(0,R))$. This implies that
$u(x)=\overline{u}(x)$ for almost all $x\in B(0,R)$, against the
fact that $\overline{u}\neq u\restriction{B_R}$. This proves the
claim. Hence $(u_n)$ weakly converges to $u\restriction{B_R}$ in
$W^{1,G}_{0}(B(0,R))$. Applying Theorem \ref{binj} again, $(u_n)$
strongly converges to $u$ in $L^{s(\cdot)} (B(0,R))$. This completes the
proof of Lemma \ref{compactbound}.
\end{proof}

Now, we are ready to prove our compact embedding result in the whole
space $\mathbb{R}^{N}$. Let us define, for every $s(\cdot)\in
C_{+}(\mathbb{R}^{N})$, the following Lebesgue space

$$L^{s(\cdot)}_{K}(\mathbb{R}^{N})=\{u:\mathbb{R}^{N} \rightarrow \mathbb{R}, \ \ \mbox{u is
measurable and} \ \ \int_{\mathbb{R}^{N}}
K(z)|u|^{s(z)}dz<+\infty\}.$$
\begin{proposition}\label{compact}
 Let $(A)$, $(G)$ and $(K)$ be
satisfied. Then $X$ is compactly embedded in
$L^{s(\cdot)}_{K}(\mathbb{R}^{N})$, for every $s(\cdot)\in
(G(\cdot),G^{\ast}(\cdot))$.
\end{proposition}
\begin{proof}
Fix $s(\cdot)\in (G(\cdot),G^{\ast}(\cdot))$ and $\epsilon>0$. It
is easy to see that
$$\displaystyle \lim_{t\rightarrow
0}\frac{|t|^{s(z)}}{|t|^{G(z)}}=\lim_{t\rightarrow
+\infty}\frac{|t|^{s(z)}}{|t|^{G^{\ast}(z)}}=0 \ \ \mbox{uniformly
for} \ \ z\in \mathbb{R}^{N}.$$ Thus, there exist $0<t_{0}<t_{1}$
and a positive constant $C>0$ such that
$$K(z)|t|^{s(z)}\leq \epsilon C(|t|^{G(z)}+|t|^{G^{\ast}(z)})+\chi_{[t_0,t_1]}(z)K(z)|t|^{G(z)} \  \mbox{for all}\ t\in \mathbb{R} \ \mbox{and} \  z\in \mathbb{R}^{N}.$$
Set $$A(u)=\displaystyle
\int_{\mathbb{R}^{N}}|u|^{G(z)}dz+\displaystyle
\int_{\mathbb{R}^{N}}|u|^{G^{\ast}(z)}dz \ \ \mbox{and} \ \ R=\{z\in
\mathbb{R}^{N}, \ \ t_{0}<|u(z)|<t_{1}\}.$$ Let $(u_{n})\in X$ be a
sequence such that $u_n\rightharpoonup u$ in $X$. It is easy to see
that $(A(u_n))_n$ is bounded in $\mathbb{R}$. Denoting $R_n=\{x\in
\mathbb{R}^{N}, \ \ t_{0}<|u_{n}(x)|<t_{1}\}$, we get $\sup_{n\in
\mathbb{N}}|A_n|<+\infty$. Hence, by $(K)$, there exists a
positive radius $r>0$ such that
\begin{align}\label{peq2}
 \displaystyle \int_{B_{r}^{c}(0)}K(z)|u_n|^{s(z)}dz &\leq \epsilon
CA(u_n)+\int_{B_{r}^{c}(0)}\chi_{[t_0,t_1]}(z)K(z)|u_n|^{G(z)}dz\nonumber \\
&\leq \epsilon
CA(u_n)+(t_{1}^{G^{-}}+t_{1}^{G^{+}})\int_{B_{r}^{c}(0)\frown
R_n}K(z)dz\nonumber\\
&\leq (C'+t_{1}^{G^{-}}+t_{1}^{G^{+}})\epsilon, \  \mbox{for all}\ n\in
\mathbb{N}.
\end{align}
Now, since $s(\cdot)\in (1,G^{\ast}(\cdot))$ and $K\in
L^{\infty}(\mathbb{R}^{N})$, we deduce, that
\begin{equation}\label{peq3}
\displaystyle \lim_{n\rightarrow +\infty}
\int_{B_{r}(0)}K(x)|u_n|^{s(z)}dz=\int_{B_{r}(0)}K(x)|u|^{s(z)}dz.
\end{equation}
Here we used Lemma \ref{compactbound}. Combining \eqref{peq2} and
\eqref{peq3}, we conclude for $\epsilon>0$ small enough, that
$$\displaystyle \lim_{n\rightarrow +\infty}
\int_{\mathbb{R}^{N}}K(z)|u_n|^{s(z)}dz=\int_{\mathbb{R}^{N}}K(z)|u|^{s(z)}dz.$$
Consequently, using Proposition \ref{pr1}, we infer that
$$u_n\rightarrow u \ \ \mbox{in} \ \ L^{s(\cdot)}_{K}(\mathbb{R}^{N}) \  \mbox{for every}\ s(\cdot)\in (G(\cdot),G^{\ast}(\cdot)).$$
This completes the proof of Proposition \ref{compact}.
\end{proof}
\section{A nonlinear problem driven by $\Delta_{G,a}$}
As an application of the previous abstract results, the main result
of this section concerns the study of both nonvariational and
singular aspects of problem \eqref{imain}.
\subsection{Nonvariational case}
In this paragraph, we work under conditions introduced in
Proposition \ref{compact}. We are mainly concerned with the
following equation
\begin{equation}\label{main1}
-\Delta_{G,a}u=-\mbox{div}(\alpha_1
u\nabla_{x}r)-\mbox{div}(\alpha_2a^{\frac{1}{G(x,y)}}(x)u\nabla_{y}r)+f(z,u),
\ \ z=(x,y)\in \mathbb{R}^{N},
\end{equation}
  The hypotheses on functions $f$ and $r$ are the following:\\
$(H_1)$  $f(z,0)\neq 0$, $f(z,s)\leq (a(z)+b(z)|s|^{\gamma(z)-1})$
and $|f(z,s)|\leq (a(z)+|b(z)||s|^{\gamma(z)-1})$
a.e. $z\in \mathbb{R}^{N}$ and for all $s\in \mathbb{R}$ where\\

$\bullet$ $\gamma(\cdot)\in C_{+}(\mathbb{R}^{N})$ and
$\gamma(\cdot),\frac{\gamma(\cdot)}{\gamma(\cdot)-1}\in
(G(\cdot),G^{\ast}(\cdot)$.

$\bullet$ $b\in C_{+}(\mathbb{R}^{N}, \mathbb{R}^{-})$ and
$\frac{b}{K}\in L^{\infty}(\mathbb{R}^{N})$.

$\bullet$ $a\in
L^{\frac{G(\cdot)}{G(\cdot)-1}}(\mathbb{R}^{N})\cap
L^{\infty}(\mathbb{R}^{N})$.

\smallskip
$(H_2)$ $r: \mathbb{R}^{N}\rightarrow \mathbb{R}$ is some measurable
function satisfying
$$ \nabla r\in L^{\frac{G(\cdot)\beta(\cdot)}{(\beta(\cdot)-1)(G(\cdot)-1)}}(\mathbb{R}^{N}), $$
where $\frac{G(\cdot)\beta(\cdot)}{G(\cdot)-1}\in (G,G^{\ast})$.\\

$(H_3)$ $\alpha_1,\alpha_2\in C_{+}(\mathbb{R}^{N})$ and
$\frac{\alpha_1}{K},\frac{\alpha_2}{K}\in
L^{\infty}(\mathbb{R}^{N})$.
\begin{definition}
We say that $u\in X\setminus\{0\}$ is a weak solution of problem \eqref{main1} if
for all $v\in X\setminus\{0\}$,
\begin{align*}
&\displaystyle\int_{\mathbb{R}^{N}}\left[\left|\nabla_{x}u\right|^{G(x,y)-2}\nabla_{x}u\nabla_{x}v+a(x)\left|\nabla_{y}u\right|^{G(x,y)-2}
     \nabla_{y}u \nabla_{y}v\right]\,dx\,dy\\
     &-\displaystyle\int_{\mathbb{R}^{N}}\alpha_1 u\nabla_{x}r.\nabla_x v
     \,dx\,dy-\displaystyle\int_{\mathbb{R}^{N}}\alpha_2 [a(x)]^{\frac{1}{G(x,y)}}u\nabla_{y}r.\nabla_y v
     \,dx\,dy\\
     &-\displaystyle\int_{\mathbb{R}^{N}} f((x,y),u)v\,dx\,dy=0.
\end{align*}
\end{definition}
\begin{remark}
Under conditions $(A)$, $(G)$, $(K)$, $(H_1)-(H_3)$ and by virtue of
Proposition \ref{compact}, the definition of weak solution of
problem \eqref{main1} is well-defined.
\end{remark}
The main result of this paragraph reads as follows.
\begin{theorem}\label{thm1}
Assume that $(G)$, $(K)$ and $(H_1)-(H_3)$ hold. Then, problem
\eqref{main1} admits at least one nontrivial weak solution.
\end{theorem}
The proof of Theorem \ref{thm1} relies on the topological degree
theory of $(S_+)$--type mappings. Define the operator $L:X\rightarrow
X^{\ast}$ by
\begin{align*}
\langle
L(u),v\rangle&=\displaystyle\int_{\mathbb{R}^{N}}\left[\left|\nabla_{x}u\right|^{G(x,y)-2}\nabla_{x}u\nabla_{x}v+a(x)\left|\nabla_{y}u\right|^{G(x,y)-2}
     \nabla_{y}u \nabla_{y}v\right]\,dx\,dy\\
     &-\displaystyle\int_{\mathbb{R}^{N}} \alpha_1 u\nabla_{x}r.\nabla_x v
     \,dx\,dy-\displaystyle\int_{\mathbb{R}^{N}} \alpha_2 [a(x)]^{\frac{1}{G(x,y)}}u\nabla_{y}r.\nabla_y v
     \,dx\,dy\\
     &-\displaystyle\int_{\mathbb{R}^{N}} f((x,y),u)v\,dx\,dy, \ \ u,v\in X.
\end{align*}
\begin{lemma}\label{SS+}
Suppose that assumptions of Theorem \ref{thm1} are fulfilled. Then
$L$ is a mapping of type $(\Ss_+$), that is, if
$u_{n}\rightharpoonup u$ in $X$ and $\displaystyle \limsup_{n\to
+\infty}\,\langle L(u_{n}),u_{n}-u\rangle_{X} \leq0$, then $u_{n}\to
u$ in $X$.
\end{lemma}
\begin{proof}
Let $\{u_{n}\}_{n \geq 1} \subseteq X$ be a sequence such that
    \begin{align*}
    u_{n}\rightharpoonup u \quad \text{in }X \quad \text{and}\quad  \displaystyle \limsup_{n\to +\infty}\,\langle L(u_{n}),u_{n}-u\rangle_X \leq 0.
    \end{align*}
    This implies that
    \begin{align}\label{eq1}
      \displaystyle \limsup_{n\to +\infty}\,\langle L(u_{n})-L(u),u_{n}-u\rangle_X \leq 0.
    \end{align}
    \textbf{Claim $1$.} $\displaystyle \lim_{n\rightarrow +\infty} \int_{\mathbb{R}^{N}} (f(z,u_n)-f(z,u))
    (u_n-u)dz=0$.\\

    For $r>0$, we denote by $B_r$ the open ball centered at the
    origin and of a radius $r$. Applying the H\"older inequality, we get
    \begin{align}\label{eq2}
\displaystyle  \int_{\mathbb{R}^{N}} &(f(z,u_n)-f(z,u))
    (u_n-u)dz \leq \int_{\mathbb{R}^{N}} (|f(z,u_n)|+|f(z,u)|)
    |u_n-u|dz \\
    &\leq  \displaystyle \int_{\mathbb{R}^{N}} |a(z)||u_n-u|dz + \displaystyle
    \int_{\mathbb{R}^{N}}
    |b(z)||u_n|^{\gamma(z)-1}|u_n-u| dz \nonumber\\ & + \displaystyle \int_{\mathbb{R}^{N}}|b(z)|
    |u|^{\gamma(z)-1}|u_n-u|dz \nonumber\\
    &\leq
    \displaystyle \int_{B_r} |a(z)||u_n-u|dz+\displaystyle \int_{B^{c}_r} |a(z)||u_n-u|dz\nonumber\\&+\||b|^{\frac{\gamma(\cdot)-1}{\gamma(\cdot)}}
    |u_n|^{\gamma(\cdot)-1}\|_{\frac{\gamma(\cdot)}{\gamma(\cdot)-1}}
    \||b|^{\frac{1}{\gamma(\cdot)}}|u_n-u|\|_{\gamma(\cdot)}\nonumber \\
   & +\||b|^{\frac{\gamma(\cdot)-1}{\gamma(\cdot)}}
    |u|^{\gamma(\cdot)-1}\|_{\frac{\gamma(\cdot)}{\gamma(\cdot)-1}}
    \||b|^{\frac{1}{\gamma(\cdot)}}|u_n-u|\|_{\gamma(\cdot)}
    \nonumber.
    \end{align}
Again, by H\"older's inequality, we obtain
$$\displaystyle \int_{B_r} |a(z)||u_n-u|dz\leq
\|a\|_{\frac{G(\cdot)}{G(\cdot)-1}(B_r)}\|u_n-u\|_{G(\cdot)}.
$$
Using Lemma \ref{compactbound}, it follows that
\begin{equation}\label{eq3}
\displaystyle \lim_{n\rightarrow +\infty}\displaystyle \int_{B_r}
|a(z)||u_n-u|dz=0.
\end{equation}
Now, using $(H_1)$, we deduce that
\begin{equation}\label{eq4}
\displaystyle \int_{B^{c}_r} |a(z)||u_n-u|dz\leq
\|a\|_{L^{\frac{G(\cdot)}{G(\cdot)-1}}(B^{c}_r)}\|u_n-u\|_{L^{G(\cdot)}(B^{c}_r)}\leq
C\|a\|_{L^{\frac{G(\cdot)}{G(\cdot)-1}}(B^{c}_r)} \rightarrow 0,
\end{equation}
as  $r \rightarrow +\infty$ and for some positive constant $C$. \\
On the other hand, by $(H_1)$ and Propositions \ref{pr1} and
\ref{compact}, we have
\begin{align*}
\||b|^{\frac{\gamma(\cdot)-1}{\gamma(\cdot)}}
    |u_n|^{\gamma(\cdot)-1}\|&_{\frac{\gamma(\cdot)}{\gamma(\cdot)-1}}
    \||b|^{\frac{1}{\gamma(\cdot)}}|u_n-u|\|_{\gamma(\cdot)}
    \leq C\||b|^{\frac{1}{\gamma(\cdot)}}|u_n-u|\|_{\gamma(\cdot)}\\
    &\leq C\left([\displaystyle
    \int_{\mathbb{R}^{N}}|b(z)||u_n-u|^{\gamma(z)}dz]^{\frac{1}{\gamma^{-}}}+[\displaystyle
    \int_{\mathbb{R}^{N}}|b(z)||u_n-u|^{\gamma(z)}dz]^{\frac{1}{\gamma^{+}}}\right)\\
    &\leq C\left([\displaystyle
    \int_{\mathbb{R}^{N}}|K(z)||u_n-u|^{\gamma(z)}dz]^{\frac{1}{\gamma^{-}}}+[\displaystyle
    \int_{\mathbb{R}^{N}}|K(z)||u_n-u|^{\gamma(z)}dz]^{\frac{1}{\gamma^{+}}}\right),
\end{align*}
for some positive constant $C$. Thus, in light of Proposition
\ref{compact}, we infer that
\begin{equation}\label{eq5}
\displaystyle \lim_{n\rightarrow +\infty}
\||b|^{\frac{\gamma(\cdot)-1}{\gamma(\cdot)}}
    |u_n|^{\gamma(\cdot)-1}\|_{\frac{\gamma(\cdot)}{\gamma(\cdot)-1}}
    \||b|^{\frac{1}{\gamma(\cdot)}}|u_n-u|\|_{\gamma(\cdot)}=0.
\end{equation}
In the same way, we prove that
\begin{equation}\label{eq6}
\displaystyle \lim_{n\rightarrow +\infty}
\||b|^{\frac{\gamma(\cdot)-1}{\gamma(\cdot)}}
    |u|^{\gamma(\cdot)-1}\|_{\frac{\gamma(\cdot)}{\gamma(\cdot)-1}}
    \||b|^{\frac{1}{\gamma(\cdot)}}|u_n-u|\|_{\gamma(\cdot)}=0
\end{equation}
Combining \eqref{eq2}, \eqref{eq3}, \eqref{eq4};\eqref{eq5} and
\eqref{eq6}, we get Claim $1$.\\

\textbf{Claim $2$}. In what follows, we show  that
\begin{align*}
\displaystyle \lim_{n\rightarrow +\infty}
\displaystyle\int_{\mathbb{R}^{N}}\alpha_1
(u_n-u)&\nabla_{x}r.\nabla_x (u_n-u)
     \,dx\,dy\\&=\displaystyle \lim_{n\rightarrow +\infty} \displaystyle\int_{\mathbb{R}^{N}}\alpha_2 [a(x)]^{\frac{1}{G(x,y)}}(u_n-u)\nabla_{y}r.\nabla_y
     (u_n-u)
     \,dx\,dy=0.
     \end{align*}
Invoking the H\"older inequality and Proposition \ref{pr1}, we obtain
\begin{align}\label{eq7}
\displaystyle\int_{\mathbb{R}^{N}}&\alpha_1
(u_n-u)\nabla_{x}r.\nabla_x (u_n-u)
     \,dx\,dy\leq
     \|\alpha_{1}|u_n-u|\nabla_xr\|_{\frac{G(\cdot)}{G(\cdot)-1)}}\|\nabla_x(u_n-u)\|_{G(\cdot)}\\
     &\leq C   \left(\displaystyle\int_{\mathbb{R}^{N}} \alpha_{1}^{\frac{G(x,y)}{G(x,y)-1}} |u_n-u|^{\frac{G(x,y)}{G(x,y)-1}}|\nabla_x r|^{\frac{G(x,y)}{G(x,y)-1}}dxdy
     \right)^{\frac{G^{-}-1}{G^{+}}}\nonumber\\ &+C\left(\displaystyle\int_{\mathbb{R}^{N}} \alpha_{1}^{\frac{G(x,y)}{G(x,y)-1}} |u_n-u|^{\frac{G(x,y)}{G(x,y)-1}}|\nabla_x r|^{\frac{G(x,y)}{G(x,y)-1}}dxdy
     \right)^{\frac{G^{+}-1}{G^{-}}} \nonumber.
\end{align}
Now, from conditions $(H_2)$, $(H_3)$ and the H\"older inequality, we
deduce that
\begin{align*}
\displaystyle\int_{\mathbb{R}^{N}} (\alpha_{1}
|u_n-u|)^{\frac{G(x,y)}{G(x,y)-1}}&|\nabla_x
r|^{\frac{G(x,y)}{G(x,y)-1}}dxdy\leq C
\displaystyle\int_{\mathbb{R}^{N}} (K
|u_n-u|)^{\frac{G(x,y)}{G(x,y)-1}}|\nabla_x
r|^{\frac{G(x,y)}{G(x,y)-1}}dxdy\\
&\leq C\|K^{\frac{1}{\beta
(\cdot)}}|u_n-u|^{\frac{G(\cdot,\cdot)}{G(\cdot,\cdot)-1}}\|_{\beta(\cdot)}\||\nabla_x
r|^{\frac{G(x,y)}{G(x,y)-1}}\|_{\frac{\beta(\cdot)}{\beta(\cdot)-1}}\\
&\leq C \|K^{\frac{1}{\beta
(\cdot)}}|u_n-u|^{\frac{G(\cdot,\cdot)}{G(\cdot,\cdot)-1}}\|_{\beta(\cdot)},
\end{align*}
which,  by Proposition \ref{compact}, implies that
\begin{equation}\label{eq8}
\displaystyle \lim_{n\rightarrow
+\infty}\displaystyle\int_{\mathbb{R}^{N}} (\alpha_{1}
|u_n-u|)^{\frac{G(x,y)}{G(x,y)-1}}|\nabla_x
r|^{\frac{G(x,y)}{G(x,y)-1}}dxdy=0.
\end{equation}
Consequently, from \eqref{eq7} and \eqref{eq8}, we conclude that
$$ \displaystyle\int_{\mathbb{R}^{N}}\alpha_1
(u_n-u)\nabla_{x}r.\nabla_x (u_n-u)
     \,dx\,dy=0.$$
Again, using the same argument, we show that
$$\displaystyle \lim_{n\rightarrow +\infty} \displaystyle\int_{\mathbb{R}^{N}}\alpha_2 [a(x)]^{\frac{1}{G(x,y)}}(u_n-u)\nabla_{y}r.\nabla_y
     (u_n-u)
     \,dx\,dy=0. $$
     This proves Claim $2$.

     Finally, from Claim $1$, Claim $2$ and \eqref{eq1}, we infer that
     $$\displaystyle \limsup_{n\to +\infty}\,\langle \rho'(u_{n})-\rho'(u),u_{n}-u\rangle_X \leq 0.$$
     Hence, by Lemma \ref{S+}, we get our desired result.
\end{proof}

\begin{lemma}\label{coercive}
Suppose that assumptions of Theorem \ref{thm1} are fulfilled. Then
for $R>0$ large enough, we have
$$\langle L(u),u\rangle>0 \ \ \mbox{for all} \ \ u\in X \ \ \mbox{such that} \ \ \|u\|=R.$$
\end{lemma}
\begin{proof}
Let $u\in X$ be such that $\|u\|>1$. Hence, in view of Lemmas
\ref{L2.2} and \ref{modular} and Proposition \ref{compact} and
the H\"older inequality, we obtain
\begin{align*}
\langle L(u),u\rangle &
=\displaystyle\int_{\mathbb{R}^{N}}\left[\left|\nabla_{x}u\right|^{G(x,y)}+a(x)\left|\nabla_{y}u\right|^{G(x,y)}
    \right]\,dx\,dy\\
     &-\displaystyle\int_{\mathbb{R}^{N}} \alpha_1 u\nabla_{x}r.\nabla_x
     u
     \,dx\,dy-\displaystyle\int_{\mathbb{R}^{N}} \alpha_2 [a(x)]^{\frac{1}{G(x,y)}}u\nabla_{y}r.\nabla_y
     u
     \,dx\,dy\\
     &-\displaystyle\int_{\mathbb{R}^{N}} f((x,y),u)v\,dx\,dy
\\&\geq
\displaystyle\int_{\mathbb{R}^{N}}\left[\left|\nabla_{x}u\right|^{G(x,y)}+a(x)\left|\nabla_{y}u\right|^{G(x,y)}
    \right]\,dx\,dy-\displaystyle\int_{\mathbb{R}^{N}} \alpha_1 u\nabla_{x}r.\nabla_x
     u
     \,dx\,dy\\
     &-\displaystyle\int_{\mathbb{R}^{N}} \alpha_2 [a(x)]^{\frac{1}{G(x,y)}}u\nabla_{y}r.\nabla_y
     u
     \,dx\,dy
     -\displaystyle\int_{\mathbb{R}^{N}}a(z)u\,dz,
     \\&
     \geq
\|u\|^{G^{-}}-\|\nabla_{x}r\|_{\frac{G(\cdot)\beta(\cdot)}{(G(\cdot)-1)(\beta(\cdot)-1)}}\|\alpha_{1}^{\frac{G(\cdot)-1}{\beta(\cdot)G(\cdot)}}u\|_{\frac{\beta(\cdot)
G(\cdot)}{G(\cdot)-1}}\|\nabla_{x}u\|_{G(\cdot)}
\\
&-
\|\nabla_{y}r\|_{\frac{G(\cdot)\beta(\cdot)}{(G(\cdot)-1)(\beta(\cdot)-1)}}\|\alpha_{2}^{\frac{G(\cdot)-1}{\beta(\cdot)G(\cdot)}}u\|_{\frac{\beta(\cdot)
G(\cdot)}{G(\cdot)-1}}\|a^{\frac{1}{G(\cdot)}}\nabla_{y}u\|_{G(\cdot)}
-
\|a\|_{\frac{G(\cdot)}{G(\cdot)-1}}\|u\|_{G(\cdot)}  \\
&\geq
\|u\|^{G^{-}}-C\|\nabla_{x}r\|_{\frac{G(\cdot)\beta(\cdot)}{(G(\cdot)-1)(\beta(\cdot)-1)}}\|K^{\frac{G(\cdot)-1}{\beta(\cdot)G(\cdot)}}u\|_{\frac{\beta(\cdot)
G(\cdot)}{G(\cdot)-1}}\|\nabla_{x}u\|_{G(\cdot)}
\\
&-C
\|\nabla_{y}r\|_{\frac{G(\cdot)\beta(\cdot)}{(G(\cdot)-1)(\beta(\cdot)-1)}}\|K^{\frac{G(\cdot)-1}{\beta(\cdot)G(\cdot)}}u\|_{\frac{\beta(\cdot)
G(\cdot)}{G(\cdot)-1}}\|a^{\frac{1}{G(\cdot)}}\nabla_{y}u\|_{G(\cdot)}
 -
\|a\|_{\frac{G(\cdot)}{G(\cdot)-1}}\|u\|_{G(\cdot)}\\
&\geq
\|u\|^{G^{-}}-C\|\nabla_{x}r\|_{\frac{G(\cdot)\beta(\cdot)}{(G(\cdot)-1)(\beta(\cdot)-1)}}\|u\|^{2}
-C
\|\nabla_{y}r\|_{\frac{G(\cdot)\beta(\cdot)}{(G(\cdot)-1)(\beta(\cdot)-1)}}\|u\|^{2} \\
&- \|a\|_{\frac{G(\cdot)}{G(\cdot)-1}}\|u\| ,
\end{align*}
where $C$ is a positive constant. Choosing $\|u\|=R$ large enough,
we deduce from the last inequality that
$$\langle L(u),u\rangle>0 \ \ \mbox{for all} \ \ u\in X \ \ \mbox{such that} \ \ \|u\|=R.$$
This completes the proof of Lemma \ref{coercive}.
\end{proof}

\textit{Proof of Theorem \ref{thm1} completed.} It is clear that $L$
is also demicontinuous and bounded. Then, in light of Lemmas
\ref{SS+} and \ref{coercive} and using the topological degree theory
for $(S_+)$ type mappings, we conclude that
$$deg (L,B(0,R), 0)=1,$$
where $R$ is defined in Lemma \ref{coercive}. Therefore the equation
$L(u) = 0$ has at least one solution $u \in B(0, R)$. From
assumption $(H_1)$, we can conclude that $u$ is a nontrivial weak
solution of equation \eqref{main1}. This completes the proof of Theorem
\ref{thm1}. \qed

\subsection{Singular problem}
In this subsection, we work under conditions introduced in
Proposition \ref{compact}. Here, we are interested in weak solutions
to nonlinear singular problems. Precisely, we study the following
singular double phase equation
\begin{equation}\label{main2}
-\Delta_{G,a}u+|u|^{G(x,y)-2}u=\frac{b(x,y)}{u^{\sigma(x,y)}}, \ \
(x,y)\in \mathbb{R}^{N},
\end{equation}
where $\sigma(\cdot) \in C^{1}(\mathbb{R}^{N}),0<\sigma(\cdot) <1$.
The assumption on
function $b$ is the following:\\
$(A)$ $b>0$ in $\mathbb{R}^{N}$, $b\in L^{1}(\mathbb{R}^{N})\bigcap
L^{G(\cdot)}(\mathbb{R}^{N})\bigcap
L^{\frac{G(\cdot)}{G(\cdot)-1}}(\mathbb{R}^{N}) $ and
$\frac{b}{K}\in L^{\infty}(\mathbb{R}^{N})$.
\begin{definition}
We say that $u \in X\setminus\{0\}$ is a weak solution of problem \eqref{main2} if
$u\geq 0, u\neq0, u^{-\sigma(\cdot)}v\in L^{1}(\mathbb{R}^{N})$ for
all $v \in X \setminus \{0\}$ and
\begin{align*}
&\displaystyle\int_{\mathbb{R}^{N}}\left[\left|\nabla_{x}u\right|^{G(x,y)-2}\nabla_{x}u\nabla_{x}v+a(x)\left|\nabla_{y}u\right|^{G(x,y)-2}
     \nabla_{y}u \nabla_{y}v\right]\,dx\,dy\\
&=\displaystyle\int_{\mathbb{R}^{N}} \frac{b(z)}{u^{\sigma(z)}}vdz.
     \end{align*}
\end{definition}
Our main result is the following existence theorem.
\begin{theorem}\label{sing}
Let $(A)$, $(G)$ and $(K)$ be satisfied. Then problem \eqref{main2}
admits at least one nontrivial positive weak solution.
\end{theorem}
To prove the above theorem, we first  consider a perturbation of
\eqref{main2} which removes the singularity. So, we consider the
following approximation of problem \eqref{main2}:
\begin{align}\label{aux}
    \begin{aligned}
    -\Delta_{G,a}u&+|u|^{G(x,y)-2}u=\frac{b(x,y)}{(u+\epsilon)^{\sigma(x,y)}}, \qquad && \ \
(x,y)\in \mathbb{R}^{N}, \\
    u>0.
    \end{aligned}
\end{align}
The main way to deal with this problem is the topological approach.
So, given $f\in L^{G(\cdot)}(\mathbb{R}^{N})$, $f\geq 0$ and
$\epsilon \in (0,1)$, we consider the following equation:
\begin{align}\label{epsaux}
    \begin{aligned}
    -\Delta_{G,a}u&+|u|^{G(x,y)-2}u=\frac{b(x,y)}{(f(x,y)+\epsilon)^{\sigma(x,y)}}, \qquad && \ \
(x,y)\in \mathbb{R}^{N}, \\
    u>0.
    \end{aligned}
\end{align}
For the above problem we have the following result.
\begin{proposition}\label{pro1}
Suppose that $(A)$, $(G)$ and $(K)$ hold. Then problem
\eqref{epsaux} admits a unique positive solution $u_{\epsilon}\in
X$.
\end{proposition}
\begin{proof}
Let $B_{G}: L^{G(\cdot)}(\mathbb{R}^{N}) \rightarrow
L^{G'(\cdot)}(\mathbb{R}^{N}) $ be the map defined by
$$B_G(u)=|u|^{G(\cdot)-2}u \ \ \ \mbox{for all} \ \ u\in L^{G(\cdot)}(\mathbb{R}^{N}).$$
Using the Simon inequality (see \cite{simon}), $B_G$ is bounded,
continuous, strictly monotone. Then we consider the map $A_G: X
\rightarrow X^{\ast }$ defined by
$$<A_G(u),v>=\int_{\mathbb{R}^{N}}\left[\left|\nabla_{x}u\right|^{G(x,y)-2}\nabla_{x}u\nabla_{x}v+a(x)\left|\nabla_{y}u\right|^{G(x,y)-2}
     \nabla_{y}u \nabla_{y}v\right]\,dx\,dy,$$
     for all $u,v\in X.$ Using the same argument, we can deduce that this operator
     is bounded continuous, strictly monotone. It follows that the operator
     $V_G=A_G+B_G$ is bounded continuous, strictly monotone (thus, maximal
     monotone, too). On the other hand, in light of Lemma \ref{S+}, we have that $V$ is coercive. We know that a maximal monotone coercive operator is surjective.
      Then, since $b(.)[f(.)+\epsilon]^{-\gamma(.)}\in
     L^{\frac{G(\cdot)}{G(\cdot)-1}}(\mathbb{R}^{N})$, we can find $v_{\epsilon}\in X$ such
     that
     \begin{equation}\label{pos}
     \langle V(v_{\epsilon}),h\rangle=\langle b(.)[f(.)+\epsilon]^{-\gamma(.)},h\rangle, \  \mbox{for every}\ h\in X.
     \end{equation}
     In \eqref{pos} we choose $h=-v_{\epsilon}^{-}$
     ($v_{\epsilon}^{-}=\max(-v_{\epsilon},0)$). Thus, using the fact that $(f(\cdot)+\epsilon)>0$, we obtain
     that $v_{\epsilon}$ is a nonnegative and $v_{\epsilon}\neq 0 $. Moreover, the strict monotonicity of $V(.)$ implies that this
     solution is unique. Finally, the anisotropic maximum principle
     of Zhang \cite{zzhang} implies that $v_{\epsilon}>0$. This completes the proof of Proposition \ref{pro1}.
\end{proof}

Using Proposition \ref{pro1}, we can define the solution map
$L_{\epsilon}: L^{G(\cdot)}(\mathbb{R}^{N})\rightarrow
L^{G(\cdot)}(\mathbb{R}^{N}) $ for problem \eqref{epsaux} by
$$L_{\epsilon}(f)=v_{\epsilon}.$$
\begin{proposition}\label{pro2}
Suppose that assumptions of Proposition \ref{pro1} are fulfilled.
Then problem \eqref{aux} admits a unique positive solution
$u_{\epsilon}\in X$.
\end{proposition}
\begin{proof}
In view of Proposition \ref{pro1}, we have
\begin{equation}\label{eq1pro2}
<A_{G}(v_{\epsilon}), h>+\displaystyle \int_{\mathbb{R}^{N}}
|v_{\epsilon}|^{G(z)-2}v_{\epsilon}hdz=\displaystyle
\int_{\mathbb{R}^{N}} b(z)[f(z) +\epsilon]^{-\gamma(z)} hdz, \
\mbox{for all}\ h\in X.
\end{equation}
In \eqref{eq1pro2} we choose $h=v_{\epsilon}=L_{\epsilon}(f)\in X$
and we obtain
$$\rho(v_{\epsilon})= \displaystyle \int_{\mathbb{R}^{N}} b(z)[f(z)
+\epsilon]^{-\gamma(z)} v_{\epsilon}dz,$$ which implies that there
exists a positive constant $C$ such that
$$\min(\|L_{\epsilon}(f)\|^{G^{-}},\|L_{\epsilon}(f)\|^{G^{+}})\leq
C_{\epsilon}\|b\|_{\frac{G(\cdot)}{G(\cdot)-1}}\|L_{\epsilon}(f)
\|$$ and
\begin{equation}\label{eq2pro2}
 \|L_{\epsilon}(f)\|\leq C_{\epsilon}, \ \ \mbox{for all} \
\ f\in L^{G(\cdot)}(\mathbb{R}^{N}).
\end{equation}

In what follows, we prove that $L_{\epsilon}(.)$ is continuous. To
this end, let $f_n \rightarrow f$ in $L^{G(\cdot)}(\mathbb{R}^{N})$.
From \eqref{eq2pro2} we have that $(L_{\epsilon}(f_{n})=u_n)_{n\in
\mathbb{N}}$ is bounded in $X$. So, we may assume that
$$u_n\rightharpoonup u \ \ \mbox{in} \ \ X.$$
Thus, using conditions $(B)$ and $(K)$, we infer that
\begin{align*}
\displaystyle \int_{\mathbb{R}^{N}} b(z)[f(z)
+\epsilon]^{-\gamma(z)} (u_n-u)dz&\leq
\frac{1}{\epsilon^{\sigma^{+}}} \displaystyle
\int_{\mathbb{R}^{N}}b^{\frac{G(z)-1}{G(z)}}(z)
b^{\frac{1}{G(z)}}(z)(u_n-u)dz\\
&\leq \frac{C}{\epsilon^{\sigma^{+}}} \displaystyle
\int_{\mathbb{R}^{N}}b^{\frac{G(z)-1}{G(z)}}(z)
K^{\frac{1}{G(z)}}(z)(u_n-u)dz\\
& \leq \frac{C}{\epsilon^{\sigma^{+}}}
\|b^{\frac{G(\cdot)-1}{G(\cdot)}}\|_{\frac{G(\cdot)}{G(\cdot)-1}}\|
K^{\frac{1}{G(\cdot)}}(u_n-u)\|_{G(\cdot)}.
\end{align*}
This leads to
\begin{equation}\label{eq3pro2}
\displaystyle \lim_{n \rightarrow +\infty}\displaystyle
\int_{\mathbb{R}^{N}} b(z)[f(z) +\epsilon]^{-\gamma(z)} (u_n-u)dz=0.
\end{equation}
Here we used Proposition \ref{compact}. On the other hand, we have
\begin{equation}\label{eq4pro2}
\langle \rho'(u_n),h\rangle=\displaystyle \int_{\mathbb{R}^{N}}
b(z)[f_n(z) +\epsilon]^{-\gamma(z)} hdz, \  \mbox{for all}\ h\in X \
\mbox{and}  \ n\in \mathbb{N}.
\end{equation}
In \eqref{eq4pro2} we choose $h=u_n-u\in X$, pass to the limit as
$n\rightarrow +\infty$ and use \eqref{eq3pro2}. Then we obtain
$$\displaystyle \lim_{n \rightarrow +\infty} \langle\rho'(u_n),u_n-u\rangle=0.$$
So, by Lemma \ref{S+},
\begin{equation}\label{eq5pro2}
u_n \rightarrow u \ \ \mbox{in} \ \ X.
\end{equation}
If in \eqref{eq4pro2} we pass to the limit as $n\rightarrow +\infty$
and use \eqref{eq5pro2}, we obtain that
$$\langle\rho'(u),h\rangle=\displaystyle \int_{\mathbb{R}^{N}}\frac{b(z)}{(f(z)+\epsilon)^{\gamma (z)}}hdz,$$
and $$L_{\epsilon}(f)=u.$$
This proves that $L_{\epsilon}(.)$ is
continuous. The continuity of  $L_{\epsilon}(.)$, together with
\eqref{eq2pro2} and Proposition \ref{compact}, permit the use of the
Schauder-Tychonov fixed point theorem (see \cite{pw}) and we find
$u_\epsilon\in X$ such  that
$L_{\epsilon}(u_{\epsilon})=u_{\epsilon}$ and so, $u_{\epsilon}$ is
a positive solution of \eqref{aux}. \\ Next we show the uniqueness
of this solution. Suppose that $v_{\epsilon}\in X$ is another
positive solution of \eqref{aux}. We have
\begin{align*}
0&\leq
\langle\rho'(u_{\epsilon})-\rho'(v_{\epsilon}),(u_{\epsilon}-v_{\epsilon})^{+}\rangle\\
&= \displaystyle
\int_{\mathbb{R}^{N}}[\frac{b(z)}{(u_{\epsilon}+\epsilon)^{\gamma(z)}}-\frac{b(z)}{(v_{\epsilon}+\epsilon)^{\gamma(z)}}](u_{\epsilon}-v_{\epsilon})^{+}dz
\leq0,
\end{align*}
which implies that $u_{\epsilon}\leq v_{\epsilon}$. Interchanging
the roles of $u_{\epsilon}$ and $v_{\epsilon}$ in the above
argument, we also have that $v_{\epsilon}\leq u_{\epsilon}$,
therefore $u_{\epsilon}=v_{\epsilon}$.
This completes the proof of Proposition \ref{pro2}.
\end{proof}

Now, we prove the following monotonicity property of the map
$\epsilon \rightarrow u_{\epsilon}$.
\begin{proposition}\label{pro3}
Assume that $(B)$, $(G)$  and $(K)$ hold. Then the map $\epsilon
\rightarrow u_{\epsilon}$ from $(0,1]$ into $X$ is nonincreasing.
\end{proposition}
\begin{proof}
Let $0<\epsilon'<\epsilon\leq 1$ and let
$u_{\epsilon},u_{\epsilon'}\in X$ be the corresponding unique
positive solutions of problem \eqref{aux}.

We define the following
function:
$$f_{\epsilon}(z,x)=\frac{b(z)}{[x^{+}+\epsilon]^{\gamma(z)}}, \ \ \mbox{if} \ \ x\leq u_{\epsilon'}(z) \ \ \mbox{and} \ \
f_{\epsilon}(z,x)=\frac{b(z)}{[u_{\epsilon'}(z)+\epsilon]^{\gamma(z)}},
\ \ \mbox{if} \ \ x> u_{\epsilon'}(z). $$ We set
$F_{\epsilon}(z,x)=\displaystyle \int_{0}^{x}f_{\epsilon}(z,s)ds$
and we introduce the functional $I_{\epsilon}:X\rightarrow
\mathbb{R}$ defined by
 \begin{align*}
 I_{\epsilon}(u)&=\displaystyle
\int_{\mathbb{R}^{N}}
\frac{|\nabla_{x}u|^{G(x,y)}}{G(x,y)}dxdy+\displaystyle
\int_{\mathbb{R}^{N}}
a(x)\frac{|\nabla_{y}u|^{G(x,y)}}{G(x,y)}dxdy+\displaystyle
\int_{\mathbb{R}^{N}} \frac{|u|^{G(x,y)}}{G(x,y)}dxdy\\&
-\displaystyle \int_{\mathbb{R}^{N}}F_{\epsilon}(z,u)dz.
 \end{align*}
  Evidently
$I_{\epsilon}$ is of class $C^{1}$. If $u\in X$ is large enough, we
have
$$
I_{\epsilon}(u)\geq
\frac{\rho(u)}{G^{-}}-\frac{\|b\|_{1}}{\epsilon^{\gamma^{+}}} \geq
\frac{\|u\|^{G^{-}}}{G^{-}}-\frac{\|b\|_{\frac{G(\cdot)}{G(\cdot)-1}}}{\epsilon^{\gamma^{+}}}.
$$
Therefore, $I_{\epsilon}$ is coercive. On the other hand, by
condition $(B)$, we can prove that $I_{\epsilon}$ is weakly lower
semicontinuous. Then, invoking the Weierstrass-Tonelli theorem, we
can find $v_{\epsilon}\in X$ such that
$$I_{\epsilon}(v_{\epsilon})=\displaystyle \inf_{u\in X}I_{\epsilon}(u).$$
This implies that
\begin{equation}\label{eq1pro3}
<\rho'(v_{\epsilon}),h>=\int_{\mathbb{R}^{N}}f_{\epsilon}(z,v_{\epsilon})hdz,
\  \mbox{for all}\ h\in X.
\end{equation}
In \eqref{eq1pro3} we choose $h=-v_{\epsilon}^{-}\in X$ and obtain
$$\rho(v_{\epsilon}^{-})=-\displaystyle \int_{\mathbb{R}^{N}}\frac{b(z)v_{\epsilon}^{-}}{\epsilon^{\gamma(z)}}dz\leq 0.$$
Hence,
$$v_{\epsilon}\geq 0, \ \ v_{\epsilon}\neq 0.$$
Now, in \eqref{eq1pro3} we choose
$h=[v_{\epsilon}-u_{\epsilon'}]^{+}\in X$. We get
$$\langle\rho'(v_{\epsilon}),(v_{\epsilon}-u_{\epsilon'})^{+}\rangle=\displaystyle \int_{\mathbb{R}^{N}}b(z)\frac{[v_{\epsilon}-u_{\epsilon'}\}^{+}}
{[u_{\epsilon'}+\epsilon]^{\gamma(z)}}dz\leq
\langle\rho'(u_{\epsilon'}),(v_{\epsilon}-u_{\epsilon'})^{+}\rangle,$$
and so
$$v_{\epsilon}\leq u_{\epsilon'}.$$
It follows, using the definition of $f_{\epsilon}(.,.)$ and
Proposition \eqref{pro2}, that $v_{\epsilon}=u_{\epsilon}$. Then,
$u_{\epsilon}\leq
u_{\epsilon'}$.
This completes the proof of Proposition \ref{pro3}.
\end{proof}

\textit{Proof of Theorem \ref{sing} completed.} Let
$(\epsilon_{n})\subseteq (0,1]$ be a sequence such that
$\epsilon_{n} \rightarrow 0^{+}$ as $n\rightarrow +\infty$ and
$u_{n}$  be as in Proposition \ref{pro2}. Then
\begin{equation}\label{eq1pro4}
\langle\rho'(u_n),h\rangle=\displaystyle
\int_{\mathbb{R}^{N}}\frac{b(z)}{[u_n+\epsilon_{n}]^{\gamma(z)}}
hdz, \ \ \mbox{for all} \ \ h\in X, \ \ \mbox{all} \ \ n\in
\mathbb{N}.
\end{equation}
In \eqref{eq1pro4} we choose $h=u_n$ and use Proposition \ref{pro3}.
Hence
$$\rho(u_n)\leq G^{+} \displaystyle
\int_{\mathbb{R}^{N}}\frac{b(z)}{u_{1}^{\gamma(z)}} u_ndz$$ which
implies that $(u_n)$ is bounded in $X$. Therefore, we can find $u\in
X$ such that
$$u_n \rightharpoonup u \ \ \mbox{in} \ \ X \ \ \mbox{and} \ \ u_n \rightarrow u \ \ a.e \ \ \mbox{in} \ \ \mathbb{R}^{N}.$$
Consequently, combining Proposition \ref{compact} and the dominated
convergence theorem, with the fact that $u_1\leq u_n$ (see
Proposition \ref{pro3}), we deduce that
\begin{equation}\label{eq2pro4}
\displaystyle \lim_{n\rightarrow +\infty} \displaystyle
\int_{\mathbb{R}^{N}}\frac{b(z)}{[u_n+\epsilon_{n}]^{\gamma(z)}}
hdz=\displaystyle \int_{\mathbb{R}^{N}}\frac{b(z)}{u^{\gamma(z)}}
hdz, \  \mbox{for every}\ h\in X.
\end{equation}
Also, it is easy to see that
\begin{equation}\label{eq3pro4}
\displaystyle \lim _{n\rightarrow
+\infty}<\rho'(u_n),h>=<\rho'(u),h>, \  \mbox{for every}  \ h\in X.
\end{equation}
Then, by  \eqref{eq2pro4} and \eqref{eq3pro4} and passing to the
limit as $n\rightarrow +\infty$ in \eqref{eq1pro4}, we conclude that
$$<A_G(u),h>+\displaystyle \int_{\mathbb{R}^{N}}|u|^{G(z)-2}u hdz=\displaystyle \int_{\mathbb{R}^{N}}\frac{b(z)}{u^{\gamma(z)}}
hdz  \ \mbox{for all}\ h\in X.$$
This proves that $u$ is a weak solution
of problem \eqref{main2}. Since $u_1\leq u_n$ for all $n\in
\mathbb{N}$, we have $u>0$. Finally, we show the uniqueness of this
positive solution. So, suppose that $v\in X$ is another positive
solution of equation \eqref{main2}. As in the proof of Proposition
\ref{pro3}, we can prove that $u=v$.
The proof of Theorem \ref{sing} is now complete. \qed

\subsection*{Acknowledgements} The authors would like to thank Professor Nikolaos S. Papageorgiou for his numerous comments and suggestions on the initial version of this paper. The research of  Vicen\c tiu D.~R\u adulescu and Du\v{s}an D. Repov\v{s} was supported by the Slovenian Research Agency program P1-0292.
The research of Vicen\c tiu D.~R\u adulescu was supported by a grant of the Ministry of Research, Innovation and Digitization, CNCS/CCCDI--UEFISCDI, project number PCE 137/2021, within PNCDI III. Du\v{s}an D. Repov\v{s} also acknowledges the support of the Slovenian Research Agency grants N1-0083 and N1-0114.


\begin{thebibliography}{99}

\bibitem{ball1} J.M. Ball, Convexity conditions and existence theorems in nonlinear elasticity, {\it
Arch. Rational Mech. Anal.} {\bf 63} (1976/77), no. 4, 337-403.

\bibitem{ball2} J.M. Ball, Discontinuous equilibrium solutions and cavitation in nonlinear elasticity, {\it
Philos. Trans. Roy. Soc. London Ser. A} {\bf 306} (1982), no. 1496,
557-611.



    \bibitem{Bahrouni-Radulescu-Repovs-2018}
    A. Bahrouni, V.D. R\u{a}dulescu, D.D. Repov\v{s}, A weighted anisotropic variant of the {C}affarelli-{K}ohn-{N}irenberg inequality and applications, {\it
    Nonlinearity} {\bf 31} (2018), no. 4, 1516--1534.

    \bibitem{Bahrouni-Radulescu-Repovs-2019}
    A. Bahrouni, V.D. R\u{a}dulescu, D.D. Repov\v{s}, Double phase transonic flow problems with variable growth: nonlinear patterns and stationary waves,
 {\it   Nonlinearity} {\bf 32} (2019), no. 7, 2481--2495.

\bibitem{Bahrouni-Radulescu-Winkert-2019}
A. Bahrouni, V.D. R\u{a}dulescu, P. Winkert, Double phase
problems with variable growth and convection for the
Baouendi-Grushin operator, {\it  Z. Angew. Math. Phys.} {\bf 71} (2020), no. 6, Paper No. 183, 15 pp.

\bibitem{Bahrouni-Radulescu-2020}
A. Bahrouni, V.D. R\u{a}dulescu, Singular double-phase systems
with variable growth for the Baouendi-Grushin operator, {\it Discrete Contin. Dyn. Syst.} {\bf 41} (2021), no. 9, 4283--4296.

\bibitem{Bahrouni-Repovs-2018}
    A. Bahrouni, D.D. Repov\v{s}, Existence and nonexistence of solutions for {$p(x)$}-curl systems arising in electromagnetism, {\it
    Complex Var. Elliptic Equ.} {\bf 63} (2018), no. 2, 292--301.

    \bibitem{Baroni-Colombo-Mingione-2015}
    P. Baroni, M. Colombo, G. Mingione, Nonautonomous functionals, borderline cases and related function classes,
 {\it   Algebra i Analiz} {\bf 27} (2015), no. 3, 6--50.

    \bibitem{baouendi} M.S. Baouendi,
     Sur une classe d'op\'erateurs elliptiques d\'eg\'en\'er\'es, {\it Bull. Soc. Math. France} {\bf 95} (1967), 45--87.

    \bibitem{beck}
     L. Beck, G. Mingione, Lipschitz bounds and non-uniform ellipticity, {\it  Comm. Pure Appl. Math.} {\bf 73} (2020), no. 5, 944--1034.

    \bibitem{Caffarelli-Kohn-Nirenberg-1984}
    L. Caffarelli, R. Kohn, L. Nirenberg, First order interpolation inequalities with weights,
   {\it Compositio Math.} {\bf 53} (1984), no. 3, 259--275.

\bibitem{cencelj}
M. Cencelj, V.D. R\u{a}dulescu, D.D. Repov\v{s}, Double phase
problems with variable growth, {\it Nonlinear Anal.}  {\bf 177} (2018),
270--287.

\bibitem{colasuonno}
F. Colasuonno, M. Squassina, Eigenvalues for double phase
variational integrals, {\it Ann. Mat. Pura Appl.} {\bf 195} (2016),
1917--1959.

     \bibitem{Colasuonno-Pucci-2011}
    F. Colasuonno, P. Pucci,
     Multiplicity of solutions for {$p(x)$}-polyharmonic elliptic {K}irchhoff equations, {\it
    Nonlinear Anal.} {\bf 74} (2011), no. 17, 5962--5974.

    \bibitem{Colombo-Mingione-2015}
    M. Colombo, G. Mingione,
     Bounded minimisers of double phase variational integrals,
  {\it  Arch. Ration. Mech. Anal.} {\bf 218} (2015), no. 1, 219--273.

    \bibitem{Edm} D. Edmunds, J. R\'akosnik, Sobolev embeddings with variable exponent, {\it Studia Math.} {\bf 143} (2000), no. 3, 267--293.

  {   \bibitem{Eleuteri} M. Eleuteri, P. Marcellini, E. Mascolo, Regularity for scalar integrals without structure conditions,
 {\it    Adv. Calc. Var.} {\bf 13} (2020), no. 3, 279--300. }

\bibitem{fan} X. Fan, J. Shen, D. Zhao, Sobolev embedding theorems for
spaces $W^{k,p(x)}(\Omega)$, {\it J. Math. Anal. Appl.} {\bf 262} (2001),
749--760.



    \bibitem{franchi}
     B. Franchi, M.C. Tesi,  A finite element approximation for a class of degenerate elliptic equations, {\it  Math. Comp.} {\bf 69} (1999), 41--63.

     \bibitem{21}  J. Giacomoni, S. Tiwari and G. Warnault, Quasilinear parabolic problem with $p(x)$-laplacian: existence,
uniqueness of weak solutions and stabilization, {\it NoDEA Nonlinear
Differential Equations Appl.} {\bf 23} (2016).

    \bibitem{grushin}
    V.V. Grushin,  On a class of hypoelliptic operators, {\it  Math. USSR-Sb.} {\bf 12} (1970), 458--476.

    \bibitem{Hajek-Montesinos-Santalucia-Vanderwerff-Zizler-2008}
    P. H\'{a}jek, V. Montesinos Santaluc\'{\i}a, J. Vanderwerff, V. Zizler,
    {\it Biorthogonal Systems in {B}anach Spaces}, Springer, New York, 2008.

 {  \bibitem{Pucci0} J. Liu, P. Pucci, H. Wu, Q. Zhang,  Existence and blow-up rate
of large solutions of $p(x)$-Laplacian equations with gradient terms, {\it J. Math. Anal. Appl.} {\bf 457} (2018), no. 1, 944--977.  }

    \bibitem{liu} W. Liu, G. Dai,  Existence and multiplicity
    results for double phase problems, {\it J. Differential Equations} {\bf
    265} (2018), 4311-4334.


\bibitem{marce1} P. Marcellini, On the definition and the lower semicontinuity of certain quasiconvex integrals, {\it Ann. Inst. H. Poincar\'e, Anal. Non Lin\'eaire} {\bf 3} (1986), 391-409.

\bibitem{marce2} P. Marcellini, Regularity and existence of solutions of elliptic equations with $p, q$--growth conditions, {\it J. Differential Equations} {\bf 90} (1991), 1-30.


 {    \bibitem{Marce} P. Marcellini,  Growth conditions and regularity for weak solutions to nonlinear elliptic pdes, {\it J. Math. Anal. Appl.}
     {\bf 501} (2021), no. 1, Paper No. 124408, 32 pp. }

      { \bibitem{Mingi}  G. Mingione, V.D. R\u adulescu, Recent developments in problems with nonstandard growth and nonuniform ellipticity, {\it J. Math. Anal. Appl.} {\bf 501} (2021), no. 1, Paper No. 125197, 41 pp. }

     \bibitem{Musielak-1983}
    J. Musielak,    {\it Orlicz Spaces and Modular Spaces}, Springer-Verlag, Berlin, 1983.

\bibitem{pw}
N.S. Papageorgiou, P. Winkert, {\it Applied Nonlinear Functional
Analysis}, De Gruyter, Berlin, 2018.

\bibitem{papageorgiou}
    N.S. Papageorgiou, V.D. R\u adulescu, D.D. Repov\v{s},
    Nonlinear singular problems with indefinite potential term,
  {\it  Anal. Math. Phys.} {\bf 9}:4 (2019), 2237-2262.

    \bibitem{prrzamp}
    N.S. Papageorgiou, V.D. R\u adulescu, D.D. Repov\v{s}, Double-phase problems with reaction of arbitrary growth, {\it Z. Angew. Math. Phys.} {\bf 69} (2018), no. 4, Art. 108, 21 pp.

    \bibitem{prrpams}
    N.S. Papageorgiou, V.D. R\u adulescu, D.D. Repov\v{s}, Double-phase problems and a discontinuity property of the spectrum, {\it Proc. Amer. Math. Soc.} {\bf 147} (2019), 2899--2910.

    \bibitem{prrbook} N.S. Papageorgiou, V.D. R\u adulescu, D.D. Repov\v{s}, {\it Nonlinear Analysis--Theory and Methods}, Springer Monographs in Mathematics, Springer, Cham, 2019.

     \bibitem{vetro}
    N.S. Papageorgiou, C. Vetro, F. Vetro, Positive solutions for singular
    $(p,2)-$equations, {\it Z. Angew. Math. Phys.} {\bf 70}:72 (2019).

  {  \bibitem{Pucci1} P. Pucci, L. Temperini, Existence for fractional $(p,q)$-systems with
critical and Hardy terms in ${\mathbb R}^N$, {\it Nonlinear Anal.} {\bf 211} (2021), 112477. }

    \bibitem{Radulescu-2015}
    V.D. R\u{a}dulescu, Nonlinear elliptic equations with variable exponent: old and new, {\it
    Nonlinear Anal.} {\bf 121} (2015), 336--369.

    \bibitem{rad2019}
    V.D. R\u{a}dulescu,  Isotropic and anisotropic double-phase problems: old and new, {\it Opuscula Math.} {\bf 39} (2019), 259--279.

\bibitem{Radulescu-Repovs-2015}
    V.D. R\u{a}dulescu, D.D. Repov\v{s},
    {\it Partial Differential Equations with Variable Exponents},
    CRC Press, Boca Raton, FL, 2015.

\bibitem{simon} J. Simon, R\'egularit\'e de la solution d'une \'equation non lin\'eaire dans ${\mathbb R}^N$, {\it
 Journ\'ees d'Analyse Non Lin\'eaire} (Proc. Conf., Besan\c{c}on, 1977), pp. 205-227, Lecture Notes in Math., 665, Springer, Berlin, 1978.

  \bibitem{zzhang}
  Q. Zhang,  A strong maximum principle for differential equations with nonstandard p(x)-growth conditions, {\it J. Math. Anal. Appl.} {\bf 312} (2005), 125-143.
    \bibitem{zhang}
     Q. Zhang, V.D. R\u adulescu,  Double phase anisotropic variational problems and combined effects of reaction and absorption terms, {\it J. Math. Pures Appl.}
    {\bf (9) 118} (2018), 159--203.

    \end{thebibliography}
\end{document}